\documentclass[a4paper,11pt]{article}

\usepackage{a4wide}
\usepackage[english]{babel}

\usepackage{amsfonts}
\usepackage{amsmath}
\usepackage{graphicx}
\usepackage[colorlinks,bookmarks=true]{hyperref}
\usepackage{amsthm}
\usepackage{enumitem}

\newtheorem{lemma}{Lemma}[section]

\newtheorem{theorem}{Theorem}[section]
\newtheorem{proposition}{Proposition}[section]

\newtheorem{assumption}{Assumption}

\newtheorem{corollary}{Corollary}[section]

\newcommand{\E}{\mathbb{E}}

\title{Rough fractional diffusions as scaling limits of nearly unstable heavy tailed Hawkes processes}

\author{Thibault Jaisson\\ CMAP, \'Ecole Polytechnique Paris \\ thibault.jaisson@polytechnique.edu\\$~~$\\
Mathieu Rosenbaum\\ LPMA, Universit\'e Pierre et Marie Curie (Paris 6)\\
  mathieu.rosenbaum@upmc.fr}

\begin{document}

\maketitle

\begin{abstract}
\noindent We investigate the asymptotic behavior as time goes to infinity of Hawkes processes whose regression kernel has $L^1$ norm close to one and power law tail of the form $x^{-(1+\alpha)}$, with $\alpha\in(0,1)$. We in particular prove that when $\alpha\in(1/2,1)$, after suitable rescaling, their law converges to that of a kind of integrated fractional Cox-Ingersoll-Ross process, with associated Hurst parameter $H=\alpha-1/2$. This result is in contrast to the case of a regression kernel with light tail, where a classical Brownian CIR process is obtained at the limit. Interestingly, it shows that persistence properties in the point process can lead to an irregular behavior of the limiting process. This theoretical result enables us to give an agent-based foundation to some recent findings about the rough nature of volatility in financial markets.
\end{abstract}


\section{Introduction}\label{intro}

\noindent A Hawkes process $(N_t)_{t\geq 0}$ is a self-exciting point process whose intensity at time $t$, denoted by $\lambda_t$, is of the form 
$$\lambda_t=\mu+\sum_{0<J_i<t}\phi(t-J_i)=\mu+\int_{(0,t)}\phi(t-s)dN_s,$$ where $\mu$ is a positive real number, $\phi$ a non-negative measurable function and the $J_i$ are the points of the process before time $t$ (see Section \ref{assumpt} for a more formal definition).
These processes have been introduced in the early seventies by Hawkes, see \cite{hawkes1971point,hawkes1971spectra,hawkes1974cluster}, in the purpose of modeling earthquakes and their aftershocks, see \cite{adamopoulos1976cluster} for such application. In the last years, the probabilistic and statistical analysis of Hawkes processes has known several interesting developments, driven by the recent use of Hawkes processes in various applied fields such as neurosciences \cite{chornoboy1988maximum,pernice2011structure,pernice2012recurrent,reynaud2013inference}, sociology \cite{blundell2012modelling,li2013dyadic,zhou2013learning}, criminology \cite{mohler2013modeling,mohler2011self}, genome analysis \cite{reynaud2010adaptive} and mostly finance \cite{ait2010modeling,bacry2013modelling,bauwens2004dynamic,bowsher2007modelling,chavez2005point,embrechts2011multivariate,errais2010affine}.\\

\noindent Among the probabilistic questions raised by Hawkes processes, particular attention has been devoted to the study of their long term scaling limits. More precisely, one wishes to understand the behavior as $T$ tends to infinity of the process 
$$\alpha_T(N_{tT}),~t\in[0,1],$$ where $\alpha_T$ is a suitable normalizing factor. In \cite{bacry2013some}, it is shown that under the condition
$$\|\phi\|_{1}=\int_0^{+\infty}\phi(s)ds< 1,$$ the asymptotic behavior of a Hawkes process is quite similar to that of a Poisson process. Indeed, as $T$ tends to infinity, 
$$\underset{t\in[0,1]}{\text{sup}}\big|\frac{N_{tT}}{T}-\E\big[\frac{N_{tT}}{T}\big]\big|\rightarrow 0,$$ in probability and 
$$\Big(\sqrt{T}\big(\frac{N_{tT}}{T}-\E\big[\frac{N_{tT}}{T}\big]\big)\Big)_{t\in[0,1]}\rightarrow \sigma (W_t)_{t\in[0,1]},$$ in law for the Skorohod topolgy, with $\sigma$ an explicit constant and $(W_t)$ a Brownian motion. This result has been extended in \cite{zhu2013nonlinear} to the case of non-linear Hawkes processes.\\

\noindent The condition $\|\phi\|_{1}<1$ is essential in order to obtain the preceding result. It is actually very similar to the assumption $|\rho|<1$ one makes on the autoregressive coefficient $\rho$ when working with a discrete time stationary AR(1) process. In particular, when starting the Hawkes process at $t=-\infty$, the assumption $\|\phi\|_{1}<1$ is required in order to get a stationary intensity with finite first moment. Also, as for AR(1) processes, under this condition, Hawkes processes only exhibit weak dependence properties. Consequently, their asymptotic behavior is in that case no surprise, close to that of a Poisson process. Hence this condition is called stability condition.\\

\noindent In \cite{jaisson2013limit}, the authors investigate the scaling limit of Hawkes processes when the stability condition is almost violated. This means
they consider a sequence of Hawkes processes satisfying the stability condition, but for which the kernel $\phi=\phi^T$ also depends on the observation scale $T$, such that $\|\phi^T\|_{1}$ tends to $1$ as $T$ goes to infinity. Such a sequence is called sequence of nearly unstable Hawkes processes.\\

\noindent Beyond its obvious mathematical interest, considering the case of nearly unstable Hawkes processes is motivated by empirical studies on financial data. Indeed, it has become quite standard to model the clustered nature of order flows on financial markets by means of Hawkes processes. However, one systematically estimates $L^1$ norms for the regression kernels which are smaller but very close to $1$, see \cite{filimonov2012quantifying,filimonov2013apparent,hardiman2013critical,lallouache2014statistically}. Interestingly, this empirical stylized fact that Hawkes processes have to be nearly unstable to fit the data has a very natural financial interpretation, namely the high degree of endogeneity of modern markets due to high frequency trading. This signifies that a large proportion of orders is just endogenously triggered by other orders, see \cite{filimonov2013apparent,hardiman2013critical,jaisson2013limit} for more details. In this framework, it is proved in \cite{jaisson2013limit} that the limiting law of a sequence of nearly unstable Hawkes processes is that of an integrated Cox-Ingersoll-Ross process (CIR process for short). Hence, compared to the case where the stability condition is in force, the asymptotic behavior at first order is no longer deterministic, see also \cite{zhu2013nonlinear} for the case where $\|\phi\|_{1}$ is exactly equal to one and other interesting developments. Note that this CIR scaling limit seems to be very consistent with financial practice. Indeed, it is widely acknowledged that there exists a linear relationship between the cumulated order flow and the integrated squared volatility, see for example \cite{wyart2008relation}, and CIR processes are very classical models for the squared volatility.\\

\noindent Nevertheless, the CIR limit in law of nearly unstable Hawkes processes discussed above is obtained under the crucial assumption
$$\int_0^{+\infty}s\phi(s)ds<+\infty.$$ It is therefore quite natural to try to extend the results of \cite{jaisson2013limit} to the case of {\it nearly unstable heavy tailed Hawkes processes}, for which this condition is no longer satisfied. Hence we consider in this paper the situation where
$$\phi(x)\underset{x\to+\infty}{\sim}\frac{K}{x^{1+\alpha}},$$
where $\alpha\in(0,1)$ and $K$ is a positive constant. This setting is actually much more in agreement with financial data, where one not only finds that the function $\phi$ has an $L^1$ norm close to one, but also that it has a power law tail, see \cite{bacry2014estimation,hardiman2013critical}. This heavy tail is quite easy to interpret in practice too: it is related to the persistence of the signed order flow (the series of $+1$, $-1$ where $+1$ represents a buy order and $-1$ a sell order). Indeed, the long memory property of this process is well established and is due to the so-called order splitting phenomenon: most orders are actually part of large orders (called metaorders), which are split in smaller orders so that prohibitive execution costs can be avoided.\\

\noindent Our main result is that for $\alpha\in(1/2,1)$, after proper rescaling, the law of a nearly unstable heavy tailed Hawkes process converges to that of a process which can be interpreted as an integrated fractional diffusion. Loosely speaking, this limiting distribution can be viewed as the integral of a fractional version of the CIR process, where a fractional Brownian motion replaces the ordinary Brownian motion. This result is quite remarkable from a probabilistic point of view. Indeed, assuming fat tail leads to a limit which is not an integrated semi-martingale. This is in strong contrast to all other scaling limits obtained for Hawkes
processes. Technically, this heavy tail case is of course more subtle than that investigated in \cite{jaisson2013limit} where semi-martingale theorems are used in a quite direct manner. Moreover, Gaussian methods are not easy to apply in our context since the limit is not a simple Gaussian functional, although it somehow involves a fractional Brownian motion.\\

\noindent The perhaps most surprising phenomenon obtained in our result is the value of the Hurst parameter $H$ of the (sort of) fractional Brownian motion appearing in the limit. Indeed, fat tail meaning persistence, one would expect getting also persistence in the limit and so $H>1/2$. This is actually the contrary: an aggregation phenomenon occurs in the heavy tail case, leading to a very irregular process in the limit, its derivative behaving as a fractional Brownian motion with Hurst parameter $H<1/2$. Coming back to financial applications, this means that in practice, the volatility process should be very irregular, which is perfectly in line with the recent empirical measures of the volatility smoothness obtained in \cite{gatheral2014volatility}. Therefore, our theoretical result shows quite clearly that the rough behavior of the volatility can be explained by the high degree of endogeneity of financial markets together with the order splitting phenomenon. This is to our knowledge the first agent-based explanation for the very rough nature of the volatility.\\

\noindent The paper is organized as follows. We first give our assumptions together with some intuitions about the limiting behavior of our processes in Section
\ref{assumpt}. Section \ref{main} contains our main theorems whose proofs can be found in Section \ref{proofs}. Finally, some technical results are relegated to an appendix.

\section{Assumptions and intuitions for the results}\label{assumpt}

We describe in this section our asymptotic framework together with intuitions about our main results which are given in Section \ref{main}.\\

\noindent We consider a sequence of point processes $(N_t^T)_{t\geq 0}$ indexed by $T$\footnote{Of course by $T$ we implicitly means $T_n$ with $n\in\mathbb N$ tending to infinity.}. For a given $T$, $(N_t^T)$ satisfies $N_0^T=0$ and the process is observed on the time interval $[0,T]$. Our asymptotic setting is that the observation scale $T$ goes to infinity. The intensity process $(\lambda_t^T)$ is defined for $t\geq 0$ by
$$\lambda_t^T=\mu^T+\int_0^t \phi^T(t-s) dN^T_s,$$
where $\mu^T$ is a sequence of positive real numbers and the $\phi^T$ are non-negative measurable functions on $\mathbb{R}^+$ which satisfies $\|\phi^T\|_{1}<+\infty$. For a given $T$, the process $(N_t^T)$ is defined on a probability space $(\Omega^T,\mathcal{F}^T,\mathbb{P}^T)$ equipped with the filtration $(\mathcal{F}_t^T)_{t\in[0,T]}$, where $\mathcal{F}_t^T$ is the $\sigma$-algebra generated by $(N_s^T)_{s\leq t}$. Moreover we assume that for any $0\leq a<b\leq T$ and $A\in \mathcal{F}_a^T$
$$\E[(N_b^T-N^T_{a})\mathrm{1}_A]=\E[\int_a^b\lambda^T_s\mathrm{1}_A ds],$$ which sets $\lambda^T$ as the intensity of $N^T$. In particular, 
if we denote by $(J_n^T)_{n\geq 1}$ the jump times of $(N_t^T)$, the process
$$N^T_{t\wedge J_n^T}-\int_0^{t\wedge J_n^T} \lambda_s^Tds$$ is a martingale and the law of $N^T$ is characterized by $\lambda^T$. From \cite{jacod1975multivariate}, such a construction can be done and the process $N^T$ is called a Hawkes process.\\

\noindent Let us now give more specific assumptions on the functions $\phi^T$. 
\begin{assumption}\label{h1}
For $t\in \mathbb{R}^+$,
$$\phi^T(t)=a_T\phi(t),$$ where $(a_T)_{T\geq 0}$ is a sequence of positive numbers converging to $1$ such that for all $T$, $a_T<1$ and $\phi$ is a non-negative measurable function such that $\|\phi\|_{1}=1$. Furthermore, 
\begin{equation*} \underset{x\rightarrow+\infty}{\emph{lim}}\alpha x^{\alpha}\big(1-F(x)\big)=K,\end{equation*}
for some $\alpha\in(0,1)$ and some positive constant $K$, with
$$F(x)=\int_0^x\phi(s)ds.$$
\end{assumption}
\noindent Recall that in \cite{jaisson2013limit}, it is assumed that \begin{equation}\label{moment}\int_0^{+\infty}t\phi(t)dt<+\infty\end{equation} and this condition leads to a CIR-type limit. Considering Assumption \ref{h1} instead of \eqref{moment} will induce a completely different scaling behavior for the sequence of nearly unstable Hawkes processes. Nevertheless, in this framework, we still have almost surely no explosion\footnote{In fact, for a Hawkes process, the no explosion property can be obtained under weaker conditions, for example $\int_0^t\phi(s)ds<\infty$ for any $t>0$, see \cite{bacry2013some}.}:
\begin{equation*}\label{explosion}\underset{n\rightarrow +\infty}{\text{lim}}J_n^T=+\infty.\end{equation*} 
Remark that we do not work in the stationary setting since our processes start at time $t=0$ and not at $t=-\infty$.\\

\noindent Let $M^T$ denote the martingale process associated to $N^T$, that is, for $t\geq 0$,
$$M^T_t=N^T_t-\int_0^t\lambda^T_sds.$$ We also set $\psi^T$ as the function defined on $\mathbb{R}^+$ by
\begin{equation}\label{defpsi}\psi^T(t)=\sum_{k=1}^\infty (\phi^T)^{*k}(t),\end{equation} where $(\phi^T)^{*1}=\phi^T$ and for $k\geq 2$, $(\phi^T)^{*k}$ denotes the convolution product of $(\phi^T)^{*(k-1)}$ with the function $\phi^T$. Note that $\psi^T(t)$ is well defined since $\|\phi^T\|_1<1$. This function plays an important role in the study of Hawkes processes, see \cite{bacry2012non}. In particular, it is proved in \cite{jaisson2013limit} that the intensity process, rescaled on $[0,1]$, can be rewritten
$$\lambda_{tT}^T=\mu^T+\int_0^{tT} \psi^T(Tt-s) \mu^T ds+\int_0^{tT} \psi^T(Tt-s)dM^T_s.$$
In term of scaling in space, a natural multiplicative factor is $(1-a_T)/\mu^T$. Indeed, in the stationary case, the expectation of $\lambda_t^T$ is $\mu^T/(1-\|\phi^T\|_1)$. Thus, the order of magnitude of the intensity is $\mu^T(1-a_T)^{-1}$. This is why we define
\begin{equation*}
C^T_t=\frac{(1-a_T)}{\mu^T}\lambda_{tT}^T.
\end{equation*}
Then we easily get 
\begin{equation}\label{decC}
C_t^T=(1-a_T)+\int_0^{t} T(1-a_T) \psi^T (Ts) ds+\sqrt{\frac{T(1-a_T)}{\mu^T}}\int_0^{t}\psi^T(T(t-s)) \sqrt{C_s^T}dB^T_s,
\end{equation} with
\begin{equation*}B_t^T=\frac{1}{\sqrt{T}}\int_0^{tT} \frac{dM_s^T}{\sqrt{\lambda^T_s}}.
\end{equation*}

\noindent From \eqref{decC}, we see that the asymptotic behavior of the intensity is closely related to that of $x\mapsto\psi^T(Tx)$. To analyze the limiting behavior of this function, let us remark that for $x\geq 0$, 
\begin{equation}\label{rhopsi}
\rho^T(x)=T\frac{\psi^T(Tx)}{\|\psi^T\|_1}\end{equation}
is the density of the random variable
$$J^T=\frac{1}{T} \sum_{i=1}^{I^T} X_i,$$
where the $(X_i)$ are iid random variables with density $\phi$ and $I^T$ is a geometric random variable with parameter $1-a_T$\footnote{$\forall k>0,~\mathbb{P}[I^T=k]=(1-a_T)(a_T)^{k-1}$.}. The Laplace transform of the random variable $J^T$, denoted by $\widehat{\rho}^T$, satisfies for $z\geq 0$

\begin{align*}
\widehat{\rho}^T(z)&=\E[e^{-zJ^T}]=\sum_{k=1}^\infty (1-a_T) (a_T)^{k-1} \E[e^{-\frac{z}{T}\sum_{i=1}^kX_i}]\\
&=\sum_{k=1}^\infty (1-a_T) (a_T)^{k-1} (\widehat{\phi}(\frac{z}{T}))^{k}=\frac{\widehat{\phi}(\frac{z}{T})}{1-\frac{a_T}{1-a_T}(\widehat{\phi}(\frac{z}{T})-1)},
\end{align*}
where $\widehat{\phi}$ denotes the Laplace of $\phi$. We now need to compute an expansion for $\widehat{\phi}(z)$. Using integration by parts, we get
$$\widehat{\phi}(z)=z\int_0^{+\infty}e^{-zt} F(t)dt=1-z\int_0^{+\infty}e^{-zt} \big(1-F(t)\big)dt.$$
Then using Assumption \ref{h1} together with Karamata Tauberian theorem (see for example Theorem 17.6 in \cite{bingham1989regular}), we get
$$\widehat{\phi}(z)=1-K\frac{\Gamma(1-\alpha)}{\alpha}z^{\alpha}+o(z^\alpha),$$
with $\Gamma$ the gamma function. Set $\delta=K\frac{\Gamma(1-\alpha)}{\alpha}$ and $v_T=\delta^{-1}T^{\alpha}(1-a_T)$. As $T$ goes to infinity, $\widehat{\rho}^T(z)$ is thus equivalent to 
\begin{equation}\label{invlap}\frac{v_T}{v_T+z^{\alpha}}.\end{equation} 
The function whose Laplace transform is equal to this last quantity
is given by 
$$
v_Tx^{\alpha-1}E_{\alpha,\alpha}(-v_Tx^\alpha),
$$
with
$E_{\alpha,\beta}$ the $(\alpha,\beta)$ Mittag-Leffler function, that is $$E_{\alpha,\beta}(z)=\sum_{n=0}^{\infty} \frac{z^n}{\Gamma(\alpha n+\beta)},$$ 
see \cite{haubold2011mittag}. Putting this together with \eqref{decC} and \eqref{rhopsi}, we can expect (for $\alpha>1/2$)
\begin{equation*}
C_t^T\sim v_T\int_0^{t}s^{\alpha-1}E_{\alpha,\alpha}(-v_Ts^\alpha)ds+\gamma_T v_T\int_0^{t}(t-s)^{\alpha-1}E_{\alpha,\alpha}(-v_T(t-s)^\alpha) \sqrt{C_s^T}dB^T_s,\end{equation*}
with $$\gamma_T=\frac{1}{\sqrt{\mu^TT(1-a_T)}}.$$
The process $B^T$ can be shown to converge to a Brownian motion $B$. Thus, denoting by $v_{\infty}$ and $\gamma_{\infty}$ the limits of $v_{T}$ and $\gamma_{T}$, passing (non rigorously) to the limit, we obtain (for $\alpha>1/2$) 
\begin{equation}\label{approx_limit}
C_t^{\infty}\sim v_{\infty}\int_0^{t}s^{\alpha-1}E_{\alpha,\alpha}(-v_{\infty}s^\alpha)ds+\gamma_{\infty} v_{\infty}\int_0^{t}(t-s)^{\alpha-1}E_{\alpha,\alpha}(-v_{\infty}(t-s)^\alpha) \sqrt{C_s^{\infty}}dB_s.\end{equation}
From \eqref{approx_limit}, we see that in order to get a non-deterministic asymptotic behavior for $C_t^T$, we need that both $v_{\infty}$ and $\gamma_{\infty}$ are positive constants or $v_{\infty}$ is equal to zero 
and $\gamma_{\infty}v_{\infty}$ is positive. However, in the last situation, the expectation of $(C_t^{\infty})^2$ would be of order $t^{2\alpha-1}$.
Such cases where the variance is increasing with power law rate are incompatible with the approximate stationarity property we want to keep for our model and its limit. Therefore, only one regime seems to be natural and this leads us to the following assumption.

\begin{assumption}
\label{speed}
There are two positive constants $\lambda$ and $\mu^*$ such that 
$$\underset{T\rightarrow +\infty}{\emph{lim}}T^\alpha(1-a_T)=\lambda\delta.$$
and
$$\underset{T\rightarrow +\infty}{\emph{lim}} T^{1-\alpha}\mu_T=\mu^*\delta^{-1}.$$
\end{assumption}

\noindent In particular, Assumption \ref{speed} implies that $v_T$ converges to $\lambda$ and therefore the sequence of random variables $(J^T)$ converges in law towards the random variable whose density on $\mathbb{R}^+$ is given by 
$$\lambda x^{\alpha-1}E_{\alpha,\alpha}(-\lambda x^\alpha).$$

\noindent Beyond giving the suitable asymptotic regimes for $a_T$ and $\mu_T$, the heuristic derivation leading to \eqref{approx_limit} provides us an expression for the limiting law of the rescaled intensities of our sequence of nearly unstable heavy tailed Hawkes processes. In \eqref{approx_limit}, this law appears under the form of a non-classic stochastic integral equation. Indeed, it is of Volterra-type and is therefore (a priori) neither a diffusion nor a semi-martingale. Furthermore, the main term of the Volterra kernel $x^{\alpha-1}$ exhibits a singularity at point $0$, of the same kind as that of the fractional Brownian motion $(B^H_t)$ when expressed under the form:
\begin{equation}\label{repfbm}
B^H_t=\frac{1}{\Gamma(H+1/2)}\big(\int_0^t(t-s)^{H-1/2}dW_s+ \int_{-\infty}^0(t-s)^{H-1/2}-(-s)^{H-1/2}dW_s\big),\end{equation}
with $(W_t)$ a Brownian motion, see \cite{mandelbrot1968fractional}.\\

\noindent The preceding computations suggest a possible approach to derive the limiting behavior of our sequence of Hawkes processes: studying the intensity of the processes. Indeed the intensities can be rewritten under the form of stochastic integral equations as \eqref{decC}. Consequently, one can try to pass to the limit in the coefficients of the equation to obtain the limiting law, as we (non rigorously) did to get \eqref{approx_limit}. This is exactly the approach used in \cite{jaisson2013limit}. However, in this more intricate heavy tail case, it seems very hard to use. In particular, the sequence $(C_t^T)$ is typically not tight. Thus, instead of considering the intensities, we directly work on the Hawkes processes themselves, more in the spirit of \cite{zhu2013nonlinear}.

\section{Main results}\label{main}

We rigorously state in this section our theorems on the limiting behavior of nearly unstable heavy tailed Hawkes processes. We start with some technical results about the function appearing as the inverse Laplace transform of \eqref{invlap} in Section \ref{assumpt}.

\subsection{The function $f^{\alpha,\lambda}$}

As shown by the derivations in the previous section, the function
$$f^{\alpha,\lambda}(x)=\lambda x^{\alpha-1}E_{\alpha,\alpha}(-\lambda x^\alpha)$$ plays a crucial role in our analysis. We give here some elements about the regularity of this function which will be useful in the sequel. We denote by $I^\alpha f$ and $D^\alpha f$ the fractional integration and derivation operators, which are defined for a suitable measurable function $f$ by
$$I^\alpha f(x)=\frac{1}{\Gamma (\alpha)}\int_0^x \frac{f(t)}{(x-t)^{1-\alpha}}dt$$
and
$$D^\alpha f(x)=\frac{1}{\Gamma (1-\alpha)}\frac{d}{dx}\int_0^x \frac{f(t)}{(x-t)^{\alpha}}dt.$$
The following lemma is a direct consequence of the definition of $f^{\alpha,\lambda}$ and Section 11 in \cite{haubold2011mittag}.

\begin{proposition}
\label{boundfalpha}
The function $f^{\alpha,\lambda}$ is $C^\infty$ on $(0,1]$ and 
$$f^{\alpha,\lambda}(x)\underset{x\to 0^+}{\sim} \frac{\lambda}{\Gamma(\alpha)}x^{\alpha-1},$$
$$(f^{\alpha,\lambda})'(x)\underset{x\to 0^+}{\sim} \frac{\lambda(\alpha-1)}{\Gamma(\alpha)}x^{\alpha-2}.$$
Furthermore, $f^{\alpha,\lambda}(x)x^{1-\alpha}$ has H\"older regularity $\alpha$ on $(0,1]$.\\

\noindent For $\nu<\alpha$, $f^{\alpha,\lambda}$ is $\nu$ fractionally differentiable and
$$D^\nu f^\alpha(x)= \lambda x^{\alpha-1-\nu}E_{\alpha,\alpha-\nu}(-\lambda x^\alpha) .$$
Therefore,
$$D^\nu f^\alpha(x)\underset{x\to 0^+}{\sim} \frac{\lambda}{\Gamma(\alpha-\nu)}\frac{1}{x^{1-\alpha+\nu}}$$
and
$$(D^\nu f^\alpha)'(x)\underset{x\to 0^+}{\sim} \frac{\lambda(\alpha-1-\nu)}{\Gamma(\alpha-\nu)}\frac{1}{x^{2-\alpha+\nu}}.$$
\noindent For $\nu'>0$, $f^\alpha$ is $\nu'$ fractionally integrable and
$$I^{\nu'} f^\alpha(x)=\lambda\frac{1}{x^{1-\alpha-\nu'}}E_{\alpha,\alpha+\nu'}(-\lambda x^\alpha).$$
Therefore,
$$I^{\nu'} f^\alpha(x)\underset{x\to 0^+}{\sim} \frac{\lambda}{\Gamma(\alpha+\nu')}\frac{1}{x^{1-\alpha-\nu'}}$$
and for $\alpha+\nu'\neq 1$, $$(I^{\nu'} f^\alpha)'(x)\underset{x\to 0^+}{\sim} \frac{\lambda(\alpha-1+\nu')}{\Gamma(\alpha+\nu')}\frac{1}{x^{2-\alpha-\nu'}}.$$
\end{proposition}
\noindent Proposition \ref{boundfalpha} will be a key tool in the proofs of the main results.

\subsection{The limiting behavior of nearly unstable heavy tailed Hawkes processes}

Let us first give some notations. We consider for $t\in[0,1]$ the renormalized Hawkes process
$$X^T_t=\frac{1-a_T}{T^{\alpha}\mu^*\delta^{-1}}N^T_{Tt}$$
and its associated integrated intensity
$$\Lambda^T_t=\frac{1-a_T}{T^{\alpha}\mu^*\delta^{-1}}\int_{0}^{tT}\lambda_s^Tds.$$
As explained in Section \ref{assumpt}, the space renormalization is chosen so that the processes have an expectation of order one.
We also introduce the martingale defined on $[0,1]$ by
$$Z_t^T=\sqrt{\frac{T^{\alpha}\mu^*\delta^{-1}}{1-a_T}}(X^T_t-\Lambda^T_t).$$
We are now ready to give our results about the convergence in distribution of $(Z^T,X^T)$ for the Skorohod topology.

\begin{proposition}\label{prop_main}
Under Assumptions \ref{h1} and \ref{speed}, the sequence $(Z^T,X^T)$ is tight. Furthermore, if $(Z,X)$ is a limit point of $(Z^T,X^T)$, then $Z$ is a continuous martingale and $[Z,Z]=X$.
\end{proposition}

\noindent Now let $(Z,X)$ be a couple of processes defined on some probability space $(\Omega,\mathcal{A},\mathbb{P})$ with law being one of the possible limit points of the sequence of distributions associated to the sequence $(Z^T,X^T)$.
From Proposition \ref{prop_main}, we are able to obtain the following theorem which is one of our main results.

\begin{theorem}\label{theo_main}
There exists a Brownian motion $B$ on $(\Omega,\mathcal{A},\mathbb{P})$ (up to extension of the space) such that
for $t\in[0,1]$, $Z_t=B_{X_t}$ and for any $\varepsilon>0$, $X$ is continuous with H\"older regularity $(1\wedge 2\alpha)-\varepsilon$ on $[0,1]$ and satisfies 
\begin{equation}
\label{eint}
X_t=\int_0^tsf^{\alpha,\lambda}(t-s)ds+\frac{1}{\sqrt{\mu^*\lambda}}\int_0^tf^{\alpha,\lambda}(t-s)B_{X_s}ds.
\end{equation}
\end{theorem}

\noindent Hence the limiting process in Theorem \ref{theo_main} has a quite original form, which can actually be interpreted more easily by looking at its derivative (when it exists).
 
\subsection{The limiting volatility process}

As explained in the introduction, when it exists, the derivative of the limiting process $X$ in Theorem \ref{theo_main} can be interpreted as a volatility function. Actually, if the tail of the function $\phi$ is not too heavy, $X$ is indeed differentiable. Let us write
$$F^{\alpha,\lambda}(t)=\int_0^t f^{\alpha,\lambda}(s)ds.$$
The following result holds.

\begin{theorem}\label{theo_deriv}
Let $(X_t)$ be a process satisfying \eqref{eint} for $t\in[0,1]$ and assume $\alpha>1/2$. Then $X$ is differentiable on $[0,1]$ and its derivative $Y$ satisfies 
\begin{equation}
Y_t=F^{\alpha,\lambda}(t)+\frac{1}{\sqrt{\mu^*\lambda}} \int_0^t f^{\alpha,\lambda}(t-s)\sqrt{Y_s}dB^1_s,\label{equavol}
\end{equation}
with $B^1$ a Brownian motion.
Furthermore, for any $\varepsilon>0$, $Y$ has H\"older regularity  $\alpha-1/2-\varepsilon$.
\end{theorem}
\subsection{Discussion}

We now comment the results given in Theorem \ref{theo_main} and Theorem \ref{theo_deriv}.\\

\noindent $\bullet~$ The singularity at zero of the function $f^{\alpha,\lambda}$ appearing in our two theorems is of order $x^{\alpha-1}$. Making an analogy with the Volterra representation of the fractional Brownian motion \eqref{repfbm}, this corresponds to a Hurst parameter $H$ equal to $\alpha-1/2$. Thus, in the case $\alpha>1/2$ where our volatility process is well defined, because of the square root term in front of the Brownian motion, we can somehow interpret \eqref{equavol} as a fractional CIR process with Hurst parameter equal to $\alpha-1/2$.  This dynamics leads to a very rough process, with H\"older regularity close to zero when $\alpha$ is close to $1/2$. As mentioned in the introduction, this is perfectly consistent with recent empirical measures of the volatility smoothness on financial data, see \cite{gatheral2014volatility}.\\

\noindent $\bullet~$ A practical consequence of the preceding point is the following: When observing on a time interval of order $$\frac{1}{(1-\|\phi\|_1)^{1/\alpha}}$$
a Hawkes process with kernel $\phi$ with $L^1$ norm close to one and power law tail with index $1+\alpha$, then after rescaling, a fractional-like behavior is obtained.\\  

\noindent $\bullet~$ Theorem \ref{theo_deriv} relates the smoothness of the volatility process to the tail parameter $\alpha$. This is particularly interesting for financial applications. Indeed, the parameter $\alpha$ is usually considered very hard to measure. Our theorem provides an approach where it can be obtained relying on the smoothness of the volatility, which is much easier to estimate, see \cite{gatheral2014volatility}.\\

\noindent $\bullet~$ The irregular volatility appearing at the limit arises because our Hawkes processes are nearly unstable with heavy tailed kernels. As explained in the introduction, in financial terms, it means that the rough behavior of the volatility can be explained by the high degree of endogeneity of modern markets combined with the persistent nature of the order flows.\\

\noindent $\bullet~$ A natural question is that of the uniqueness of the solution of Equation \eqref{equavol}. Because of the singular Volterra kernel and of the square root term, it is probably quite difficult to answer. Actually, in the very recent paper \cite{mytnik2015uniqueness}, using SPDE techniques, the authors show weak uniqueness for the solutions of an equation which is quite similar to \eqref{equavol}. However, they use the specific form of their equation and their approach cannot be adapted to our case in an obvious way.\\

\noindent $\bullet~$ Note that Hawkes processes with $L^1$ norm exactly equal to one have been introduced in \cite{bremaud2001hawkes}. In this work, the authors show that in order to get a stationary intensity, the parameter $\mu$ must be equal to zero and the regression kernel has to be heavy tailed. Several additional results for the non-stationary heavy tailed case ($\|\phi\|_1=1$ and $\mu>0$) can be found in \cite{zhu2013nonlinear}.\\

\noindent $\bullet~$ Compared to the approach in \cite{jaisson2013limit}, it is important to remark that our volatility process is simply the derivative of the limit of the sequence of nearly unstable heavy tailed Hawkes processes. Contrary to what is done in \cite{jaisson2013limit}, we do not provide any result about the convergence of the sequence of intensities of the Hawkes processes. In particular, the sequence of intensities is not shown to converge towards the volatility\footnote{Actually it can be shown that for some reasonable functions $\phi$, the sequence of intensities does not converge, at least in the Skorohod topology.}. Note also that our assumptions are slightly weaker than those in \cite{jaisson2013limit}. In particular, we do not require the function $\phi$ to be bounded. Again, this is relevant for financial applications where $\phi(t)$ becomes typically very large as $t$ tends to zero, see \cite{bacry2014estimation}.

\section{Proofs}\label{proofs}

We give in this section the proofs of Proposition \ref{prop_main}, Theorem \ref{theo_main} and Theorem \ref{theo_deriv}. In the sequel, $c$ denotes a positive constant which may vary from line to line (and even within the same line if no ambiguity).

\subsection{Proof of Proposition \ref{prop_main}}
We show here the tightness of $(Z^T,X^T)$. We start with the following lemma.

\begin{lemma}\label{ctight}
The sequences $X^T$ and $\Lambda^T$ are $\mathbb{C}$-tight.
\end{lemma}

\begin{proof}
From \cite{bacry2013some}, we get that the expectation of the Hawkes process $N^T_t$ satisfies
$$\E[N^T_t]=\mu_Tt+\mu_T\int_0^t\psi^T(t-s)sds\leq t\mu_T(1+\|\psi^T\|_{1}).$$
Therefore, since 
$$\|\psi^T\|_{1}\leq \frac{c}{1-a_T},$$
we get
$$\E[X^T_1]=\E[\Lambda^T_1]\leq c.$$
The tightnesses of $X^T$ and $\Lambda^T$ then follow, using the fact that both processes are increasing.\\

\noindent Moreover, since $(1-a_T)/T^{\alpha}$ tends to zero, the maximum jump size of $X^T$ and $\Lambda^T$ (which is continuous) goes to zero as $T$ tends to infinity. From Proposition VI-3.26 in \cite{jacod2003limit}, this implies the $\mathbb{C}$-tightness of $X^T$ and $\Lambda^T$.
\end{proof}

\noindent We now give the proof of Proposition \ref{prop_main}. It is easy to get that the angle bracket of $Z^T$ is $\Lambda^T$. From Lemma \ref{ctight}, it is $\mathbb{C}$-tight. Thus, from Theorem VI-4.13 in \cite{jacod2003limit}, the sequence $(Z^T)$ is tight. Finally, marginal tightnesses imply the joint tightness of $(Z^T,X^T)$.\\

\noindent Let us now consider a subsequence $(Z^{T_n},X^{T_n})$ converging towards a process that we denote by $(Z,X)$. Using Proposition VI-6.26 in \cite{jacod2003limit} together with the fact that the bracket of $(Z^{T_n})$ is $(X^{T_n})$, we get that $X=[Z,Z]$.\\

\noindent Since $\sqrt{\frac{1-a_T}{T^{\alpha}}}$ goes to $0$, the maximum jump size of $Z^{T_n}$ tends to zero. Therefore, $Z^{T_n}$ is $\mathbb{C}$-tight and so the limit $Z$ is continuous. It remains to show that $Z$ is a martingale. Using Corollary IX.1.19 in \cite{jacod2003limit}, $Z$ is a local martingale. Moreover, the expectation of its bracket being finite, it is a martingale.

\subsection{Proof of Equation \eqref{eint} in Theorem \ref{theo_main}}

We start with the following lemma which shows that we can somehow work with $\Lambda^T$ rather than with $X^T$.

\begin{lemma}\label{samelim}
The sequence of martingales $X^T-\Lambda^T$ tends to zero in probability, uniformly on $[0,1]$.
\end{lemma}

\begin{proof}
We have
$$X^T_t-\Lambda^T_t=\frac{1-a_T}{T^{\alpha}\mu^*\delta^{-1}}M^T_{tT}.$$
Applying Doob's inequality to the martingale $M^T$, we get
\begin{equation*}\label{doob}
\E[\sup_{t\in[0,1]}\{(X^T_t-\Lambda^T_t)^2 \}]\leq c(\frac{1-a_T}{T^{\alpha}})^2 \E[(M^T_T)^2].
\end{equation*}
Then, the bracket of $M^T$ being $N^T$, we deduce
$$\E[\sup_{t\in[0,1]}\{(X^T_t-\Lambda^T_t)^2 \}]\leq c(\frac{1-a_T}{T^{\alpha}})^2 \E[N^T_T]\leq c\frac{\mu_T(1-a_T)}{T^{2\alpha}}\leq c\frac{1-a_T}{T^{\alpha}},$$
which ends the proof.
\end{proof}
\noindent We now state a lemma which will be useful in the proof of Equation \eqref{eint}.

\begin{lemma}\label{ucp}
The sequence of measures with density $\rho^T(x)$ defined by Equation \eqref{rhopsi} converges weakly towards the measure with density $\lambda x^{\alpha-1}E_{\alpha,\alpha}(-\lambda x^{\alpha})$. In particular, over $[0,1]$,
$$F^T(t)=\int_0^t \rho^T(x)dx$$
converges uniformly towards
$$F^{\alpha,\lambda}(t)=\int_0^t f^{\alpha,\lambda}(x)dx.$$
\end{lemma}
\begin{proof}
The proof of this result is obtained by showing that the Laplace transform of the measure with density $\rho^T(x)$ converges towards the Laplace transform of the measure with density $\lambda x^{\alpha-1}E_{\alpha,\alpha}(-\lambda x^{\alpha})$. This has already been done in Section \ref{assumpt}.
\end{proof}

\noindent We now give the proof of Equation \eqref{eint}. Let us consider a converging subsequence $(Z^{T_n},X^{T_n})$ and write $(Z,X)$ its limit. Abusing notation slightly, we write $(Z^T,X^T)$ instead of $(Z^{T_n},X^{T_n})$. Using Skorokhod's representation theorem, there exists a probability space on which one can define copies in law of the $(Z^T,X^T)$ converging almost surely for the Skorohod topology to a random variable with the same law as $(Z,X)$. We now work with this sequence of variables converging almost surely and their limit. The processes $Z$ and $X$ being continuous, we have
\begin{equation}\label{ucpxz}
\sup_{t\in [0,1]} |X^T_t-X_t|\rightarrow 0,~\sup_{t\in [0,1]} |Z^T_t-Z_t|\rightarrow 0.\end{equation}

\noindent Let us now rewrite the cumulated intensity. For all $t\geq 0$, we have
$$\int_0^{t} \lambda^T_sds=t\mu_T+\int_0^t \phi^T(t-s)(\int_0^s \lambda^T_u du)ds+\int_0^t\phi^T(t-s) M^T_sds.$$
Then, using that $\psi^T \ast \phi^T=\psi^T-\phi^T$, where $\psi^T$ is defined in Equation \eqref{defpsi}, remark that
\begin{align*}
\int_0^t \psi^T(t-s) \int_0^s \phi^T(s-r) M^T_r dr ds&= \int_0^t \int_0^t \mathrm{1}_{r\leq s}  \psi^T(t-s) \phi^T(s-r)dsM^T_r dr\\
 &= \int_{0}^t \int_0^{t-r} \psi^T(t-r-s)  \phi^T(s)ds M^T_rdr\\
&= \int_{0}^t  \psi^T \ast \phi^T(t-r) M^T_rdr\\
&= \int_{0}^t  \psi^T (t-r) M^T_rdr-\int_{0}^t  \phi^T (t-r) M^T_rdr.
\end{align*}
This together with Lemma 3 in \cite{bacry2013some} yields
$$\int_0^{t} \lambda^T_sds=t\mu_T+\int_0^t \psi^T(t-s)s\mu_Tds+\int_0^t\psi^T(t-s) M^T_sds.$$
Therefore, replacing $t$ by $Tt$, multiplying by $(1-a_T)/(T^{\alpha}\mu^*\delta^{-1})$, and writing $$u_T=\frac{\mu_T}{\mu^*\delta^{-1}T^{\alpha-1}},$$ we get
$\Lambda^T(t)=T_1+T_2+T_3,$ with
\begin{align*}
T_1&=(1-a_T)tu_T,\\
T_2&=T(1-a_T)u_T\int_0^{t}\psi^T(T(t-s))s ds,\\
T_3&=T^{1-\alpha/2}\sqrt{\frac{(1-a_T)}{\mu^*\delta^{-1}}}\int_0^t\psi^T(T(t-s)) Z^T_sds.
\end{align*}
Since $u_T$ converges to $1$, we get that $T_1$ goes to zero. For $T_2$, note that integrating by parts, we have
$$T_2=a_Tu_T\int_0^{t}\rho^T(t-s)s ds=a_Tu_T\int_0^{t}F^T(t-s) ds.$$
Using Lemma \ref{ucp} and integrating by parts again, we obtain that $T_2$ tends uniformly to 
$$\int_0^{t}F^{\alpha,\lambda}(t-s) ds=\int_0^{t}f^{\alpha,\lambda}(t-s)s ds.$$

\noindent We now turn to $T_3$. Remark that
$$T_3=\frac{a_T}{\sqrt{T^{\alpha}(1-a_T)\mu^*\delta^{-1}}}\int_0^t\rho^T(t-s) Z^T_sds$$ 
and recall that
$$Z_t^T=\sqrt{\frac{T^{\alpha}\mu^*\delta^{-1}}{1-a_T}}(X^T_t-\Lambda^T_t).$$
Thus, using that $X^T$ is piecewise constant, applying integration by parts, we get (pathwise)
$$\int_0^t\rho^T(t-s) Z^T_sds=\int_0^tF^T(t-s)dZ^T_s$$
and in the same way
$$\int_0^tf^{\alpha,\lambda}(t-s) Z^T_sds=\int_0^tF^{\alpha,\lambda}(t-s)dZ^T_s.$$
Then,
$$\E\Big[\big(\int_0^t(F^{\alpha,\lambda}(t-s)-F^T(t-s))dZ^{T}_s\big)^2\Big]\leq c\int_0^t\big(F^{\alpha,\lambda}(t-s)-F^T(t-s)\big)^2d s,$$
which tends to zero thanks to Lemma \ref{ucp}. Furthermore, using \eqref{ucpxz}, we get that
$$\int_0^t\lambda(t-s)^{\alpha-1}E_{\alpha,\alpha}(-\lambda (t-s)^{\alpha})|Z_s-Z_s^T|ds$$ also tends to zero.
Consequently, we finally obtain that for any $t$, $T_3$ converges to
$$\frac{1}{\sqrt{\mu^*\lambda}}\int_0^t\lambda(t-s)^{\alpha-1}E_{\alpha,\alpha}(-\lambda (t-s)^{\alpha})Z_sds.$$

\noindent Since $Z$ is a continuous martingale, the fact that $Z_t=B_{X_t}$ is a consequence of the Dambis-Dubin-Schwarz theorem, see for example Theorem V-1.6 in \cite{revuz1999continuous}.

\subsection{Proof of the H\"older property for $X$ in Theorem \ref{theo_main}}

We start with the following lemma.

\begin{lemma}
\label{rec}
Let $B$ be a Brownian motion and $X$ a solution of \eqref{eint} associated to $B$. Let $H$ in $(0,1)$. If $X$ has H\"older regularity $H$ on $[0,1]$, then for any $\varepsilon>0$, $X$ has also H\"older regularity $((\alpha+H/2)\wedge 1)-\varepsilon$ on $[0,1]$.
\end{lemma}

\begin{proof}Let $\varepsilon>0$ and $Z_t=B_{X_t}$.
The function
$$t\rightarrow \int_0^tsf^{\alpha,\lambda}(t-s)ds$$ being $\mathcal{C}^1$, it is enough to show that
$$t\rightarrow\int_0^tf^{\alpha,\lambda} (t-s)Z_sds$$ has H\"older regularity $((\alpha+H/2)\wedge 1)-\varepsilon$. Since for any $\varepsilon'>0$, $Z$ has H\"older regularity $(H/2-\varepsilon')$, by Proposition \ref{cfdif}, it is $(H/2-\varepsilon)$ fractionally differentiable and $D^{H/2-\varepsilon} Z$ is continuous. Using the fact that $f^{\alpha,\lambda}$ is fractionally integrable, from Corollary \ref{fipp1}, we get
$$\int_0^tf^{\alpha,\lambda} (t-s)Z_sds=\int_0^tI^{H/2-\varepsilon} f^{\alpha,\lambda} (t-s)D^{H/2-\varepsilon}Z_sds.$$
Finally, the properties of $I^{H/2-\varepsilon} f^{\alpha,\lambda}$ stated in Proposition \ref{boundfalpha} together with Proposition \ref{power} give the result.
\end{proof}

\noindent We now show that for $B$ be a Brownian motion and $X$ a solution of \eqref{eint} associated to $B$, then, for any $\varepsilon>0$, almost surely, the process $X$ has H\"older regularity $(1\wedge 2\alpha)-\varepsilon$ on $[0,1]$.\\

\noindent Let $M$ be the supremum of the H\"older exponents of $X$. From Proposition \ref{boundfalpha} together with Proposition \ref{power}, we get that $M\geq \alpha$.\\

\noindent Let us now assume that $M<(1\wedge 2 \alpha)$. Then we can find some $H<M$ and some $\varepsilon>0$ such that 
$$M<\big((\alpha+H/2)\wedge 1\big)-\varepsilon.$$ Thus, since $X$ has H\"older regularity $H$, Lemma \ref{rec} implies that $X$ has also H\"older regularity
$$\big((\alpha+H/2)\wedge 1\big)-\varepsilon,$$ which is a contradiction. Therefore $M\geq (1\wedge 2 \alpha)$, which ends the proof.

\subsection{Proof of Theorem \ref{theo_deriv}}

First remark that thanks to the H\"older property of the Brownian motion together with that of the process $X$, we immediately deduce the following lemma.

\begin{lemma}
\label{holder}
Let $B$ be a Brownian motion, $X$ a solution of \eqref{eint} associated to $B$ and $Z_t=B_{X_t}$. Then, for any $\varepsilon>0$, almost surely, the process $Z$ has H\"older regularity $(1/2\wedge \alpha)-\varepsilon$ on $[0,1]$.
\end{lemma}



\noindent We now give the proof of Theorem \ref{theo_deriv}. Using Proposition \ref{boundfalpha}, Lemma \ref{holder} and Corollary \ref{fipp2}, for any $\nu\in(0,\alpha)$, we can rewrite Equation \eqref{eint} as
$$X_t=\int_0^tsf^{\alpha,\lambda}(t-s)ds+\frac{1}{\sqrt{\mu^*\lambda}} \int_0^t D^\nu f^{\alpha,\lambda}(t-s)  I^\nu Z_sds.$$
Moreover, taking $\nu>1/2$, since $Z$ is $1-\nu$ fractionally differentiable, we get $$I^\nu Z_s=\int_0^s D^{1-\nu}Z_udu.$$
Thus, using Fubini's theorem, we obtain
\begin{align*}
\int_0^t D^\nu f^{\alpha,\lambda}(t-s)  I^\nu Z_sds&=\int_0^t \int_0^s D^\nu f^{\alpha,\lambda}(t-s) D^{1-\nu}Z_udu ds\\
&=\int_0^t \int_u^t D^\nu f^{\alpha,\lambda}(t-s) D^{1-\nu}Z_u ds du\\
&=\int_0^t \int_u^t D^\nu f^{\alpha,\lambda}(s-u) D^{1-\nu}Z_u ds du\\
&=\int_0^t \int_0^s D^\nu f^{\alpha,\lambda}(s-u) D^{1-\nu}Z_u du ds.
\end{align*}
Hence, we get
$$X_t=\int_0^t Y_s ds,$$
with $$Y_s=F^{\alpha,\lambda}(s)+\frac{1}{\sqrt{\mu^*\lambda}}\int_0^s D^\nu f^{\alpha,\lambda}(s-u)  D^{1-\nu} Z_udu.$$
From Proposition \ref{boundfalpha} together with Proposition \ref{power}, we have that $Y$ has H\"older regularity $(\alpha-\nu)$. Thus, taking $\nu$ close enough to $1/2$, we get that for any $\varepsilon>0$, $Y$ has H\"older regularity $(\alpha-1/2-\varepsilon)$. This implies that $X$ is differentiable with derivative $Y$.\\

\noindent Now, since $Z$ is a continuous martingale with bracket $X$ and because $\nu>1/2$, we can use the stochastic Fubini theorem, see for example \cite{veraar2012stochastic}, to obtain 
\begin{align*}
D^{1-\nu} Z_s&= \frac{1}{\Gamma(\nu)}\frac{d}{ds}\int_0^s \frac{Z_v}{(s-v)^{1-\nu}}dv\\
&= \frac{1}{\Gamma(\nu)}\frac{d}{ds}\int_0^s\int_0^v \frac{1}{(s-v)^{1-\nu}}dZ_udv\\
&= \frac{1}{\Gamma(\nu)}\frac{d}{ds}\int_0^s\int_u^s \frac{1}{(s-v)^{1-\nu}}dvdZ_u\\
&=\frac{1}{\Gamma(\nu+1)}\frac{d}{ds}\int_0^s (s-u)^\nu dZ_u.
\end{align*}

\noindent Therefore,
$$Y_t=F^{\alpha,\lambda}(t)+\frac{1}{\sqrt{\mu^*\lambda}}\int_0^t D^\nu f^\alpha(t-s)  \frac{1}{\Gamma(\nu+1)}\frac{d}{ds}\int_0^s (s-u)^\nu dZ_uds.$$
Using Fubini's theorem twice and the fact that $f\ast (g')=(f\ast g)'$, we derive
\begin{align*}
Y_t&=F^{\alpha,\lambda}(t)+\frac{1}{\sqrt{\mu^*\lambda}} \frac{d}{dt}\int_0^t \frac{1}{\Gamma(\nu+1)} \int_u^tD^\nu f^\alpha(t-s)(s-u)^\nu ds  dZ_u\\
&= F^{\alpha,\lambda}(t)+\frac{1}{\sqrt{\mu^*\lambda}} \frac{d}{dt}\int_0^t I^{\nu+1} D^\nu f^\alpha(t-u)  dZ_u \\
&= F^{\alpha,\lambda}(t)+\frac{1}{\sqrt{\mu^*\lambda}} \frac{d}{dt}\int_0^t \int_0^{v}I^{\nu} D^\nu f^\alpha(v-u)  dZ_udv \\
&=F^{\alpha,\lambda}(t)+\frac{1}{\sqrt{\mu^*\lambda}} \int_0^t  f^\alpha(t-u)  dZ_u.
\end{align*}
Moreover, using Theorem V-3.8 of \cite{revuz1999continuous}, there exists a Brownian motion $B^1$ such that
$$Z_t=\int_0^t\sqrt{Y_s} dB^1_s.$$
So, consider now the process $(\tilde Y_t)$ defined by 
$$\tilde Y_t=F^{\alpha,\lambda}(t)+\frac{1}{\sqrt{\mu^*\lambda}} \int_0^t  f^\alpha(t-u) \sqrt{Y_s} dB^1_s.$$
Going backward in the previous computations for $Y_t$ and $D^{1-\nu} Z_s$, we remark that
$$\tilde Y_t=F^{\alpha,\lambda}(t)+\frac{1}{\sqrt{\mu^*\lambda}}\int_0^t D^\nu f^\alpha(t-s)  \frac{1}{\Gamma(\nu+1)}\frac{d}{ds}\int_0^s (s-u)^\nu \sqrt{Y_u} dB^1_uds$$
and
\begin{align*}
\frac{1}{\Gamma(\nu+1)}\frac{d}{ds}\int_0^s (s-u)^\nu \sqrt{Y_u} dB^1_u&= \frac{1}{\Gamma(\nu)}\int_0^s\int_0^v \frac{1}{(s-v)^{1-\nu}} \sqrt{Y_u} dB^1_udv\\
&=\frac{1}{\Gamma(\nu)}\frac{d}{ds}\int_0^s \frac{1}{(s-v)^{1-\nu}} \big(\int_0^v\sqrt{Y_u} dB^1_u\big)dv\\
&= \frac{1}{\Gamma(\nu)}\frac{d}{ds}\int_0^s \frac{Z_v}{(s-v)^{1-\nu}}dv\\
&=D^{1-\nu} Z_s.
\end{align*}
Therefore,
$$\tilde Y_t=F^{\alpha,\lambda}(t)+\frac{1}{\sqrt{\mu^*\lambda}}\int_0^t D^\nu f^\alpha(t-s) D^{1-\nu} Z_sds=Y_t.$$
Consequently,
$$Y_t=F^{\alpha,\lambda}(t)+\frac{1}{\sqrt{\mu^*\lambda}} \int_0^t  f^\alpha(t-u) \sqrt{Y_s} dB^1_s.$$
\appendix

\section{Technical appendix}

In this section, we gather some useful results from \cite{samko1993fractional} and recall a theorem on the regularity of the convolution product. We denote by $H^\lambda$ the set of functions on $[0,1]$ with H\"older regularity $\lambda$.

\subsection{Fractional integrals and derivatives}

\noindent Lemma 13.1 in \cite{samko1993fractional} relates the H\"older exponent of a function and the H\"older exponent of its fractional derivatives.

\begin{proposition}
\label{cfdif}
If $f \in H^\lambda$ and $f(0)=0$, then for any $\alpha<\lambda$, $f$ admits a fractional derivative of order $\alpha$ and $D^\alpha f \in H^{\lambda-\alpha}$.
\end{proposition}

\noindent Equation 2.20 in \cite{samko1993fractional} is a fractional integration by parts formula which can be written as follows.

\begin{proposition}
\label{fipp}
If $\phi\in L^p$ and $\psi\in L^q$ with $1/p+1/q\leq 1+\alpha$, then $\phi$ and $\psi$ have an integral of order $\alpha$ and
$$\int_0^t\phi (t-s) I^\alpha \psi(s)ds=\int_0^tI^\alpha\phi (t-s)  \psi(s)ds.$$
\end{proposition}

\noindent In this work, we mainly use the two following corollaries of Proposition \ref{fipp}.

\begin{corollary}
\label{fipp1}
Let $\phi\in L^r$, with $r>1$ and $\psi\in H^\beta$. Then, for any $\alpha<\beta$, $D^\alpha \psi$ exists, belongs to $ H^{\beta-\alpha}$ and
$$\int_0^t\phi (t-s)  \psi(s)ds=\int_0^tI^\alpha\phi (t-s)  D^\alpha \psi(s)ds.$$
\end{corollary}

\begin{corollary}
\label{fipp2}
Let $\phi$ be continuous and $\psi$ such that $x^\mu\psi(x)\in H^{\lambda}$ for some $\mu>0$. Then, for any $\alpha<\min(1-\mu,\lambda)$, $D^\alpha \psi$ exists, belongs to $L^r$ for some $r>1$ and
$$\int_0^t\phi (t-s)  \psi(s)ds=\int_0^tI^\alpha\phi (t-s)  D^\alpha \psi(s)ds.$$
\end{corollary}

\subsection{Convolution}

\noindent Finally, the next result is about the smoothness of the convolution of a power type function with a continuous function.

\begin{proposition}
\label{power}
Let $f$ be a differentiable function on $(0,1]$ such that for some $K>0$, $0<\beta<1$ and any $x$ in $(0,1]$, $$|f(x)| \leq \frac{K}{x^{\beta}}~and~|f'(x)| \leq \frac{K}{x^{\beta+1}},$$
and $g$ a continuous function on $[0,1]$. Then the convolution
$$f\ast g(t)=\int_0^t f(t-s)g(s)ds $$
has H\"older regularity $(1-\beta)$.
\end{proposition}

\begin{proof}
We write $G$ for the supremum of $|g|$ and we split $f\ast g(t+h)-f\ast g(t)$ into the three following terms:
\begin{eqnarray*}
f\ast g(t+h)-f\ast g(t)&=& \int_{t}^{t+h} f(t+h-s)g(s)ds \\
&+& \int_{t-h}^{t} \big(f(t+h-s)-f(t-s)\big)g(s)ds\\
&+& \int_{0}^{t-h} \big(f(t+h-s)-f(t-s)\big)g(s)ds.\\
\end{eqnarray*}
The first term is bounded by $KG\frac{h^{1-\beta}}{1-\beta}$, the second by
$KG(1+\frac{1}{1-\beta})h^{1-\beta}$ and the third by
$$G \int_{0}^{t-h} \int_{t-s}^{t+h-s}|f'(u)| duds\leq GK\int_{0}^{t-h} h \frac{1}{(t-s)^{1+\beta}}ds\leq \frac{2}{\beta}GK h^{1-\beta}.$$
\end{proof}

\section*{Acknowledgements}
We thank Jean Jacod and Nakahiro Yoshida for helpful discussions.

\bibliographystyle{abbrv}
\bibliography{bibli_HQ_final}

\begin{thebibliography}{10}

\bibitem{adamopoulos1976cluster}
L.~Adamopoulos.
\newblock Cluster models for earthquakes: Regional comparisons.
\newblock {\em Journal of the International Association for Mathematical
  Geology}, 8(4):463--475, 1976.

\bibitem{ait2010modeling}
Y.~A{\"\i}t-Sahalia, J.~Cacho-Diaz, and R.~J. Laeven.
\newblock Modeling financial contagion using mutually exciting jump processes.
\newblock {\em Journal of Financial Economics}, to appear, 2013.

\bibitem{bacry2012non}
E.~Bacry, K.~Dayri, and J.-F. Muzy.
\newblock Non-parametric kernel estimation for symmetric {H}awkes processes.
  {A}pplication to high frequency financial data.
\newblock {\em The European Physical Journal B}, 85(5):1--12, 2012.

\bibitem{bacry2013modelling}
E.~Bacry, S.~Delattre, M.~Hoffmann, and J.-F. Muzy.
\newblock Modelling microstructure noise with mutually exciting point
  processes.
\newblock {\em Quantitative Finance}, 13(1):65--77, 2013.

\bibitem{bacry2013some}
E.~Bacry, S.~Delattre, M.~Hoffmann, and J.-F. Muzy.
\newblock Some limit theorems for {H}awkes processes and application to
  financial statistics.
\newblock {\em Stochastic Processes and their Applications}, 123(7):2475--2499,
  2013.

\bibitem{bacry2014estimation}
E.~Bacry, T.~Jaisson, and J.-F. Muzy.
\newblock Estimation of slowly decreasing {H}awkes kernels: Application to high
  frequency order book modelling.
\newblock {\em Working paper}, 2014.

\bibitem{bauwens2004dynamic}
L.~Bauwens and N.~Hautsch.
\newblock Dynamic latent factor models for intensity processes.
\newblock {\em CORE Discussion Paper}, 2004.

\bibitem{bingham1989regular}
N.~H. Bingham, C.~M. Goldie, and J.~L. Teugels.
\newblock {\em Regular variation}.
\newblock Cambridge university press, 1989.

\bibitem{blundell2012modelling}
C.~Blundell, J.~Beck, and K.~A. Heller.
\newblock Modelling reciprocating relationships with {H}awkes processes.
\newblock In {\em Advances in Neural Information Processing Systems}, pages
  2600--2608, 2012.

\bibitem{bowsher2007modelling}
C.~G. Bowsher.
\newblock Modelling security market events in continuous time: Intensity based,
  multivariate point process models.
\newblock {\em Journal of Econometrics}, 141(2):876--912, 2007.

\bibitem{bremaud2001hawkes}
P.~Br{\'e}maud and L.~Massouli{\'e}.
\newblock {H}awkes branching point processes without ancestors.
\newblock {\em Journal of Applied Probability}, 38(1):122--135, 2001.

\bibitem{chavez2005point}
V.~Chavez-Demoulin, A.~C. Davison, and A.~J. McNeil.
\newblock A point process approach to value-at-risk estimation.
\newblock {\em Quantitative Finance}, 5(2):227--234, 2005.

\bibitem{chornoboy1988maximum}
E.~Chornoboy, L.~Schramm, and A.~Karr.
\newblock Maximum likelihood identification of neural point process systems.
\newblock {\em Biological Cybernetics}, 59(4-5):265--275, 1988.

\bibitem{embrechts2011multivariate}
P.~Embrechts, T.~Liniger, and L.~Lin.
\newblock Multivariate {H}awkes processes: an application to financial data.
\newblock {\em Journal of Applied Probability}, 48:367--378, 2011.

\bibitem{errais2010affine}
E.~Errais, K.~Giesecke, and L.~R. Goldberg.
\newblock Affine point processes and portfolio credit risk.
\newblock {\em SIAM Journal on Financial Mathematics}, 1(1):642--665, 2010.

\bibitem{filimonov2012quantifying}
V.~Filimonov and D.~Sornette.
\newblock Quantifying reflexivity in financial markets: Toward a prediction of
  flash crashes.
\newblock {\em Physical Review E}, 85(5):056108, 2012.

\bibitem{filimonov2013apparent}
V.~Filimonov and D.~Sornette.
\newblock Apparent criticality and calibration issues in the {H}awkes
  self-excited point process model: application to high-frequency financial
  data.
\newblock {\em arXiv preprint arXiv:1308.6756}, 2013.

\bibitem{gatheral2014volatility}
J.~Gatheral, T.~Jaisson, and M.~Rosenbaum.
\newblock Volatility is rough.
\newblock {\em arXiv preprint arXiv:1410.3394}, 2014.

\bibitem{hardiman2013critical}
S.~J. Hardiman, N.~Bercot, and J.-P. Bouchaud.
\newblock Critical reflexivity in financial markets: a {H}awkes process
  analysis.
\newblock {\em The European Physical Journal B}, 86(10):1--9, 2013.

\bibitem{haubold2011mittag}
H.~J. Haubold, A.~M. Mathai, and R.~K. Saxena.
\newblock Mittag-{L}effler functions and their applications.
\newblock {\em Journal of Applied Mathematics}, 2011.

\bibitem{hawkes1971point}
A.~G. Hawkes.
\newblock Point spectra of some mutually exciting point processes.
\newblock {\em Journal of the Royal Statistical Society. Series B
  (Methodological)}, pages 438--443, 1971.

\bibitem{hawkes1971spectra}
A.~G. Hawkes.
\newblock Spectra of some self-exciting and mutually exciting point processes.
\newblock {\em Biometrika}, 58(1):83--90, 1971.

\bibitem{hawkes1974cluster}
A.~G. Hawkes and D.~Oakes.
\newblock A cluster process representation of a self-exciting process.
\newblock {\em Journal of Applied Probability}, pages 493--503, 1974.

\bibitem{jacod1975multivariate}
J.~Jacod.
\newblock Multivariate point processes: predictable projection,
  {R}adon-{N}ikodym derivatives, representation of martingales.
\newblock {\em Probability Theory and Related Fields}, 31(3):235--253, 1975.

\bibitem{jacod2003limit}
J.~Jacod and A.~N. Shiryaev.
\newblock {\em Limit theorems for stochastic processes}.
\newblock Springer-Verlag, Berlin, 2003.

\bibitem{jaisson2013limit}
T.~Jaisson and M.~Rosenbaum.
\newblock Limit theorems for nearly unstable {H}awkes processes.
\newblock {\em The Annals of Applied Probability}, 25(2):600--631, 2015.

\bibitem{lallouache2014statistically}
M.~Lallouache and D.~Challet.
\newblock Statistically significant fits of {H}awkes processes to financial
  data.
\newblock {\em Available at SSRN 2450101}, 2014.

\bibitem{li2013dyadic}
L.~Li and H.~Zha.
\newblock Dyadic event attribution in social networks with mixtures of {H}awkes
  processes.
\newblock In {\em Proceedings of the 22nd ACM international conference on
  Conference on information \& knowledge management}, pages 1667--1672. ACM,
  2013.

\bibitem{mandelbrot1968fractional}
B.~B. Mandelbrot and J.~W. Van~Ness.
\newblock Fractional brownian motions, fractional noises and applications.
\newblock {\em SIAM Review}, 10(4):422--437, 1968.

\bibitem{mohler2013modeling}
G.~Mohler et~al.
\newblock Modeling and estimation of multi-source clustering in crime and
  security data.
\newblock {\em The Annals of Applied Statistics}, 7(3):1525--1539, 2013.

\bibitem{mohler2011self}
G.~O. Mohler, M.~B. Short, P.~J. Brantingham, F.~P. Schoenberg, and G.~E. Tita.
\newblock Self-exciting point process modeling of crime.
\newblock {\em Journal of the American Statistical Association}, 106(493),
  2011.

\bibitem{mytnik2015uniqueness}
L.~Mytnik and T.~S. Salisbury.
\newblock Uniqueness for volterra-type stochastic integral equations.
\newblock {\em arXiv preprint arXiv:1502.05513}, 2015.

\bibitem{pernice2011structure}
V.~Pernice, B.~Staude, S.~Cardanobile, and S.~Rotter.
\newblock How structure determines correlations in neuronal networks.
\newblock {\em PLoS Computational Biology}, 7(5):e1002059, 2011.

\bibitem{pernice2012recurrent}
V.~Pernice, B.~Staude, S.~Cardanobile, and S.~Rotter.
\newblock Recurrent interactions in spiking networks with arbitrary topology.
\newblock {\em Physical Review E}, 85(3):031916, 2012.

\bibitem{revuz1999continuous}
D.~Revuz and M.~Yor.
\newblock {\em Continuous martingales and Brownian motion}.
\newblock Springer, 1999.

\bibitem{reynaud2013inference}
P.~Reynaud-Bouret, V.~Rivoirard, and C.~Tuleau-Malot.
\newblock Inference of functional connectivity in neurosciences via {H}awkes
  processes.
\newblock In {\em 1st IEEE Global Conference on Signal and Information
  Processing}, 2013.

\bibitem{reynaud2010adaptive}
P.~Reynaud-Bouret and S.~Schbath.
\newblock Adaptive estimation for {H}awkes processes; application to genome
  analysis.
\newblock {\em The Annals of Statistics}, 38(5):2781--2822, 2010.

\bibitem{samko1993fractional}
S.~G. Samko, A.~A. Kilbas, and O.~I. Marichev.
\newblock Fractional integrals and derivatives.
\newblock {\em Theory and Applications, Gordon and Breach, Yverdon}, 1993.

\bibitem{veraar2012stochastic}
M.~Veraar.
\newblock The stochastic {F}ubini theorem revisited.
\newblock {\em Stochastics An International Journal of Probability and
  Stochastic Processes}, 84(4):543--551, 2012.

\bibitem{wyart2008relation}
M.~Wyart, J.-P. Bouchaud, J.~Kockelkoren, M.~Potters, and M.~Vettorazzo.
\newblock Relation between bid--ask spread, impact and volatility in
  order-driven markets.
\newblock {\em Quantitative Finance}, 8(1):41--57, 2008.

\bibitem{zhou2013learning}
K.~Zhou, H.~Zha, and L.~Song.
\newblock Learning social infectivity in sparse low-rank networks using
  multi-dimensional {H}awkes processes.
\newblock In {\em Proceedings of the Sixteenth International Conference on
  Artificial Intelligence and Statistics}, pages 641--649, 2013.

\bibitem{zhu2013nonlinear}
L.~Zhu.
\newblock Nonlinear {H}awkes processes.
\newblock {\em arXiv preprint arXiv:1304.7531}, 2013.

\end{thebibliography}

\end{document}